\documentclass[a4paper,12pt,intlimits,oneside]{amsart}
\usepackage{amsmath}
\usepackage{amssymb}
\usepackage{amsfonts}
\usepackage{amsthm,amscd}

\numberwithin{equation}{section}
\newtheorem{theorem}{Theorem}[section]

\newtheorem{definition}[theorem]{Definition}

\newtheorem{lemma}[theorem]{Lemma}

\newtheorem{proposition}[theorem]{Proposition}
\newtheorem{remark}[theorem]{Remark}

\newcommand{\comment}[1]{}

\numberwithin{equation}{section}

\def\lsim{\raisebox{-1ex}{$~\stackrel{\textstyle <}{\sim}~$}}

\theoremstyle{definition}

\begin{document}
\title []{Boundedness of the Hardy-Littlewood maximal operator on generalized Fofana spaces}

\author[P. Nagacy]{Pokou Nagacy}
\address{Laboratoire des Sciences et Technologies de l'Environnement, UFR Environnement, Universit\'e Jean Lorougnon
GUEDE, BP 150 Daloa, C\^ote d'Ivoire}
\email{{pokounagacy@yahoo.com}}

\author[B. A. Kpata]{B\'erenger Akon Kpata}
\address{Laboratoire de Math\'ematiques et Informatique, UFR Sciences Fondamentales et
Appliquées,
 Université Nangui Abrogoua,
 02 BP 801 Abidjan 02,  C\^ote d'Ivoire}
\email{{kpata\_akon@yahoo.fr}}
\author[N. Diarra]{Nouffou Diarra}
\address{Laboratoire de Math\'ematiques et Applications, UFR Math\'ematiques et Informatique,
 Universit\'e F\'elix
 Houphou\"et-Boigny,
 22 BP 582 Abidjan 22,  C\^ote d'Ivoire}
\email{{nouffoud@yahoo.fr}}

 \renewcommand{\thefootnote}{}
\footnote{\emph{Corresponding author}: Pokou Nagacy}

\footnote{\emph{Email address:} pokounagacy@yahoo.com}

\footnote{2020 \emph{Mathematics Subject Classification}: 42B25, 42B35.}

\footnote{\emph{Key words and phrases}: Amalgam spaces, generalized Fofana spaces, generalized Morrey spaces, Hardy-Littlewood maximal operator}

\renewcommand{\thefootnote}{\arabic{footnote}}
\setcounter{footnote}{0}

\begin{abstract}
We introduce generalized Fofana spaces and we give some of their basic properties. These spaces are a kind of generalization of generalized Morrey spaces. As application, we establish the boundedness of the Hardy-Littlewood maximal operator on these new spaces.
\end{abstract}

\maketitle

\section{Introduction}
In this paper, the classical Lebesgue space on the Euclidean space $\mathbb{R}^{d}$ is denoted by $L^{p}(\mathbb{R}^{d})$ and $\|\cdot \|_{p} $ stands for its usual norm. For any measurable subset $A$ of $\mathbb{R}^{d}$, we denote by $\vert A \vert$ and $\chi_{A}$ its Lebesgue measure and its characteristic function, respectively. For any $x \in \mathbb{R}$ and for any real number $r>0$, $B(x,\, r)$ stands for the ball centered in $x$ with radius $r$. We denote by $ L^{0}(\mathbb{R}^{d}) $, the complex vector space of equivalent classes (modulo equality Lebesgue almost everywhere) of Lebesgue measurable complex-valued functions on $ \mathbb{R}^{d}$.

The Hardy-Littlewood maximal opertor $M$ is defined, for all locally integrable functions $f$ on $\mathbb{R}^{d}$, by
$$
 Mf(x)=\sup_{r>0}\frac{1}{\vert B(x,r)\vert}\int_{B(x,r)}\left|f(y)\right|dy,\;\;\;\;\; x\in \mathbb{R}^{d}.
$$
Its boundedness property on Banach spaces, including Lebesgue spaces, Morrey and generalized Morrey spaces, and Fofana spaces, was widely investigated (see, for instance \cite{CF}, \cite{GF},  \cite{JF}, \cite{Sa-Faz-Ha}, \cite{EN}, \cite{S}). Although the Hardy-Littlewood maximal function is a classical tool in harmonic analysis, it has recently found applications in the studies of Sobolev functions and partial differential equations (see \cite{BH} and \cite{L}).

Let us recall that the classical Morrey spaces $ \mathcal{M}^{q, \lambda}(\mathbb{R}^{d}) \: (0\leq \lambda \leq d, 1\leq q \leq \infty)$ were introduced in \cite{Mo} by C.B. Morrey. They are generalizations of Lebesgue spaces and are related to regularity results for solutions of certain problems in partial differential equations and integral operators (see \cite{Sa-Faz-Ha} and the references therein). Morrey spaces also belong to the larger class of Fofana spaces $\left(L^{q},L^{p}\right)^{\alpha}(\mathbb{R}^d)$ $ (1 \leq q \leq \alpha \leq p \leq \infty) $ defined as follows:
$$(L^{q},L^{p})^{\alpha}(\mathbb{R}^{d})= \left\{
      \begin{array}{ll}
      f \in L^{0}(\mathbb{R}^{d}) : \quad  \Vert f \Vert_{q,p, \alpha } < \infty
      \end{array}
      \right\},$$
 where
 $$ \Vert f \Vert_{q,p, \alpha }= r^{d(\frac{1}{\alpha}-\frac{1}{q}-\frac{1}{p})}\:_{r}\left\|f\right\|_{q,p},$$
 with 
\begin{equation}
 _{r}\left\|f\right\|_{q,p} = \left\|\left[\int_{\mathbb R^{d}}\vert f\chi_{B(y,r)}\vert^{q}(x)dx\right]^{\frac{1}{q}}\right\|_p . \label{M}
 \end{equation} 
Here, the $ L^{p}(\mathbb{R}^{d})-$norm is taken with respect to the variable $ y$ and we adopt the convention $\frac{1}{\infty} =0$. The spaces $\left(L^{q},L^{p}\right)^{\alpha}(\mathbb{R}^d)$ were first introduced by Fofana with ''discrete'' norms, which are denoted in this paper by $ \left\|\cdot\right\|_{(L^{q},l^{p})^{\alpha}(\mathbb{R}^{d})} $ (equivalent to $ \Vert \cdot \Vert_{q,p, \alpha}$ ), see \cite{IF1}. We recall that for any locally integrable function $f$,  
 $$\left\| f \right\|_{(L^{q},l^{p})^{\alpha}(\mathbb{R}^{d})} = r^{d(\frac{1}{\alpha}-\frac{1}{q})}\:_{r}\widetilde{\left\|f\right\|}_{q,p},$$
where 
$$_{r}\widetilde{\left\|f\right\|}_{q,p}= \left\|\left\{  \left\|f\chi_{Q^{r}_{k}}\right\|_{q} \right\}_{k \in \mathbb{Z}^{d} }\right\|_{\ell^{p}},  $$
with $k=(k_{1},\, ...,\,k_{d})\in \mathbb{Z}^{d}$ and
 $ Q^{r}_{k} = \prod_{i=1}^{d}  [rk_{i}, r(k_{i}+1)) $.\\
 It is proved in \cite{IF1} that the following assertions hold:\\
 $ \bullet $ $\left(\left(L^{q},L^{p}\right)^{\alpha}(\mathbb{R}^d),\, \Vert \cdot \Vert_{q,p, \alpha }\right)$ is a Banach space; \\
 $ \bullet $ for $1 \leq q \leq \alpha < \infty $ fixed, $\{\left(L^{q},L^{p}\right)^{\alpha}(\mathbb{R}^d)\; :\, \alpha\leq p \leq \infty \}$ is a non-decreasing family (in the sense of inclusion) of distinct Banach spaces whose minimal and maximal elements are the Lebesgue space $L^{\alpha}(\mathbb{R}^d) = \left(L^{\alpha},L^{p}\right)^{\alpha}(\mathbb{R}^d)= \left(L^{q},L^{\alpha}\right)^{\alpha}(\mathbb{R}^d) $ and the classical Morrey space $ \left(L^{q},L^{\infty}\right)^{\alpha}(\mathbb{R}^d) = \mathcal{M}_{q, \lambda}(\mathbb{R}^{d})$, respectively, where $\lambda= \frac{1}{q} -\frac{1}{\alpha}$.\\
The spaces $ \left(L^{q},L^{p}\right)^{\alpha}(\mathbb{R}^d) $ have been successfully used in extending several important results known in the setting of Lebesgue and Morrey spaces (see, for example \cite{DFS}, \cite{JF2}, \cite{JF}, \cite{KFK}, \cite{MF}).

In recent decades, many works have been devoted to the concept of generalized Morrey spaces and their applications (see, for example \cite{Sa} and the references therein). The generalized Morrey spaces $\mathcal{M}_{q,\, \phi}$, where $1\leq q \leq \infty $ and $\phi : (0,\,\infty) \to (0, \,\infty)$ is a function satisfying certain conditions, were introduced by Nakai in \cite{EN} to investigate the boundedness properties of some classical operators of harmonic analysis.

In the present work, we introduce a kind of generalization of Nakai's generalized Morrey spaces, called generalized Fofana spaces. We give some of their basic properties. As application, we establish the boundedness of the Hardy-Littlewood maximal operator in these new spaces.

This paper is organized as follows. In Section 2, we define generalized Fofana spaces and we give some of their basic properties. In Section 3, we state and prove the boundedness property of the Hardy-Littlewood maximal operator on generalized Fofana spaces.\\
For $A,B>0$, the notation $A \lsim B$ means that there exists a constant $C>0$ such that $A \leq C B$, where $C$ depends only on the parameters of importance. If $ A \lsim B$ and $B \lsim A$, then we write $ A\approx  B$.

\section{Generalized Fofana spaces}
 Let $1\leq q,\, p \leq\infty$ and $r>0$. The amalgam space $(L^{q},\, l^{p})_{r}(\mathbb{R}^{d})$ is defined as the set of all functions $f \in L^{0}(\mathbb{R}^{d})$ satisfying $ \: _{r}\widetilde{\left\|f\right\|}_{q,p}<\infty$. \\
  Amalgam spaces were introduced by N. Wiener \cite{W} in 1926. But its systematic study goes back to the work
  of Holland \cite{FH}. For more informations about amalgam spaces, we refer to \cite{JF}, \cite{FS2} and the references cited therein. We recall some properties of these spaces in the following proposition.
 \begin{proposition} \label{prop 3}
Let $1\leq q,\, p \leq\infty$ and $r>0$. Then, the following assertions
hold.
\begin{enumerate}
\item  There exists a constant $ C>0 $ such that
$$C^{-1}\:r^{\frac{d}{p}}\:_{r}\widetilde{\left\|f\right\|}_{q,p}\leq \: _{r}\left\|f\right\|_{q,p} \leq C r^{\frac{d}{p}}\:_{r}\widetilde{\left\|f\right\|}_{q,p},\;\;\;\;\; f \in L^{0}(\mathbb{R}^{d}).$$
\item  $((L^{q},\, l^{p})_{r}(\mathbb{R}^{d}),   \: _{r}\widetilde{\left\|\cdot  \right\|}_{q,p}) $ is a complex Banach space equals to \\ $((L^{q},\, L^{p})_{r}(\mathbb{R}^{d}),   \: _{r}\left\|\cdot  \right\|_{q,p}) .$ 
\item If  $1\leq q_{1}\leq q_{2}\leq \infty$ and $1\leq p \leq \infty$, then for all $f \in L^{0}(\mathbb{R}^{d})$,
 $$ _{r}\widetilde{\left\|f\right\|}_{q_{1},p} \leq \: _{r}\widetilde{\left\|f\right\|}_{q_{2},p} r^{d({\frac{1}{q_{1}} - \frac{1}{q_{2}}})}.$$
\item If $1\leq p_{1}\leq p_{2}\leq\infty$ and $1\leq q \leq \infty$, then for all $f \in L^{0}(\mathbb{R}^{d})$, 
  $$_{r}\widetilde{\left\| f \right\|}_{q,p_{2}} \leq 2^{d} \: _{r}\widetilde{\left\|f\right\|}_{q,p_{1}}.$$
\end{enumerate}
 \end{proposition}

\begin{definition}
Let $1\leq q, p\leq \infty$. We denote by $\mathcal{G}_{q,p}$ the set of measurable functions $\phi : (0,\,\infty) \to (0, \,\infty)$ that satisfy the following assertions: 
 \begin{itemize}
\item $ t \mapsto t^{\frac{d}{p}} \phi(t) $ is almost decreasing (there exists $ C > 0$ such that $ r \leq s \implies s^{\frac{d}{p}}\phi(s)  \leq C r^{\frac{d}{p}}\phi(r) $);
\item $ t \mapsto t^{\frac{d}{q}} \phi(t) $ is almost increasing (there exists $ C > 0$ such that $ r \leq s \implies r^{\frac{d}{q}}\phi(r)  \leq C s^{\frac{d}{q}}\phi(s)$).
 \end{itemize}
\end{definition}

\begin{remark}
Note that:
\begin{enumerate}
\item if $\phi \in \mathcal{G}_{q,p}$ then the following 
doubling condition is fulfilled:
 \begin{equation}\label{eq 1}
 \exists \: C>0, \quad \quad  \frac{1}{2} \leq \frac{r}{s} \leq 2 \implies \frac{1}{C} \leq \frac{\phi (r)}{\phi(s)} \leq C;
 \end{equation}
 \item if $ q_{1}\leq q_{2}$ then $ \mathcal{G}_{q_{2},p} \subset \mathcal{G}_{q_{1},p}; $
 \item if $ p_{1}\leq p_{2}$ then $ \mathcal{G}_{q,p_{1}} \subset \mathcal{G}_{q,p_{2}}.$
\end{enumerate}
\end{remark}

 \begin{definition}
 Let $1\leq q \leq p\leq\infty$ and let $\phi \in \mathcal{G}_{q,p}$. The generalized Fofana space $ (L^{q}, L^{p})^{\phi} (\mathbb{R}^{d}) $ is defined as 
   $$(L^{q},L^{p})^{\phi}(\mathbb{R}^{d})= \left\{
     \begin{array}{ll}
     f \in L^{0}(\mathbb{R}^{d}) : \quad  \Vert f \Vert_{(L^{q},L^{p})^{\phi}(\mathbb{R}^{d}) } < \infty
     \end{array}
     \right\},$$
  where 
  \begin{equation*}
  \left\|f\right\|_{(L^{q},L^{p})^{\phi}(\mathbb{R}^{d})}=\sup_{r>0} \frac{1}{\phi(r)} r^{d(-\frac{1}{q}-\frac{1}{p})}\: _{r}\left\|f\right\|_{q,p}.
  \end{equation*}
 \end{definition}
 \begin{remark}
Note that:
\begin{enumerate}
\item the space $(L^{q}, L^{p})^{\phi} (\mathbb{R}^{d})$ is a complex vector subspace of $ L^{0}(\mathbb{R}^{d})$;
\item the map $ L^{0}(\mathbb{R}^{d}) \ni f\mapsto\ \left\|f\right\|_{(L^{q},L^{p})^{\phi}(\mathbb{R}^{d})}$ defines a norm on $(L^{q},L^{p})^{\phi}(\mathbb{R}^{d})$;
\item according to Proposition \ref{prop 3} (1),
$$ \left\|f\right\|_{(L^{q},L^{p})^{\phi}(\mathbb{R}^{d})}\approx \widetilde{\left\|f\right\|}_{(L^{q},L^{p})^{\phi}(\mathbb{R}^{d})}, \quad \quad f \in (L^{q}, L^{p})^{\phi} (\mathbb{R}^{d}),$$
where $$\widetilde{\left\|f\right\|}_{(L^{q},L^{p})^{\phi}(\mathbb{R}^{d})}=\sup_{r>0} \frac{1}{\phi(r)} r^{-\frac{d}{q}}\:_{r}\widetilde{\left\|f\right\|}_{q,p};$$
\item if  $1\leq q \leq \alpha \leq  p $ and $\phi (r) = r^{-\frac{d}{\alpha}}$, then $ (L^{q}, L^{p})^{\phi} (\mathbb{R}^{d})$ is the Fofana space $(L^{q}, L^{p})^{\alpha} (\mathbb{R}^{d})$;
\item for $1\leq q <\infty$, the space   $(L^{q},L^{\infty})^{\phi}(\mathbb{R}^{d})$ is the generalized Morrey space $\mathcal{M}_{q, \phi}$ defined in \cite{EN}.
\end{enumerate}
\end{remark}

\begin{proposition}
 Let $1\leq q \leq p\leq \infty$ and $\phi \in \mathcal{G}_{q,p}$. Then \\ $\left((L^{q},L^{p})^{\phi}(\mathbb{R}^{d}),  \left\|\cdot\right\|_{(L^{q},L^{p})^{\phi}(\mathbb{R}^{d})}\right)$, is a complex Banach space.
 \end{proposition}

\begin{proof}
Let $ (f_{m})_{m \in \mathbb{N}} $ be a Cauchy sequence of elements of $ (L^{q},L^{p})^{\phi}(\mathbb{R}^{d})$. \\ Fix $\epsilon > 0$. Then there exists an integer $ N $ such that for all integers $ m $ and $ n $ we have
 	$$  m,n > N \implies \Vert f_{m} - f_{n} \Vert_{(L^{q},L^{p})^{\phi}(\mathbb{R}^{d})}  < \epsilon. $$
 This implies that for $ r $ fixed and for $  m,n > N $,
 		$$\frac{1}{\phi(r)} r^{d(-\frac{1}{q}-\frac{1}{p})}\: _{r}\Vert f_{m} - f_{n} \Vert_{q,p}  < \epsilon. $$
 Hence, $ (f_{m})_{m \in \mathbb{N}} $ is a Cauchy sequence of $ \left(L^{q},L^{p}\right)_{r}(\mathbb{R}^{d}) $.	As $ \left(L^{q},L^{p}\right)_{r}(\mathbb{R}^{d}) $ is a complex
 Banach space, $ (f_{m})_{m \in \mathbb{N}} $ converges to $f \in (L^{q}, L^{p})_{r} (\mathbb{R}^{d})$.  In addition, for all $ m,n > N $ we
 have
 		$$ \vert \Vert f_{m} \Vert_{(L^{q},L^{p})^{\phi}(\mathbb{R}^{d})} - \Vert f_{n} \Vert_{(L^{q},L^{p})^{\phi}(\mathbb{R}^{d})} \vert \leq \Vert f_{m} - f_{n} \Vert_{(L^{q},L^{p})^{\phi}(\mathbb{R}^{d})} < \epsilon.$$
 Therefore, $(\Vert f_{n} \Vert_{(L^{q},L^{p})^{\phi}(\mathbb{R}^{d})})_{n \in \mathbb{N}}$ is a Cauchy sequence of $ \mathbb{R}$. Let us put $M \ge 0$ its limit. It follows that for $ m > N $,
 	\begin{eqnarray*}
 & &	\frac{1}{\phi(r)} r^{d(-\frac{1}{q}-\frac{1}{p})}\: _{r}\left\|f\right\|_{q,p}\\
 		& \leq & \Vert f_{m} \Vert_{(L^{q},L^{p})^{\phi}(\mathbb{R}^{d})} + \frac{1}{\phi(r)} r^{d(-\frac{1}{q}-\frac{1}{p})} \:_{r}\Vert f - f_{m} \Vert_{q,p}.
 	\end{eqnarray*}
 	By letting $ m $ go to $ \infty $, we get
  $$ \frac{1}{\phi(r)} r^{d(-\frac{1}{q}-\frac{1}{p})}\ _{r}\left\|f\right\|_{q,p} \leq M.$$
  Hence, $\Vert f \Vert_{(L^{q},L^{p})^{\phi}(\mathbb{R}^{d})} \leq M$ and $f \in  (L^{q},L^{p})^{\phi}(\mathbb{R}^{d})$. \\
 Let $ m $ be an integer such that $m>N$. For all $r>0 $ and for all $l \in \mathbb{N}$, we have
 \begin{eqnarray*}
 \frac{1}{\phi(r)} r^{d(-\frac{1}{q}-\frac{1}{p})} \:_{r}\Vert f_{m} -f \Vert_{q,p}
 & \leq &  \Vert f_{m} - f_{m+l}  \Vert_{(L^{q},L^{p})^{\phi}(\mathbb{R}^{d})} \\
 & & + \frac{1}{\phi(r)} r^{d(-\frac{1}{q}-\frac{1}{p})} \:_{r}\Vert f_{m+l} -f \Vert_{q,p}\\
 &\leq& \epsilon + \frac{1}{\phi(r)} r^{d(-\frac{1}{q}-\frac{1}{p})} \:_{r}\Vert f_{m+l} -f \Vert_{q,p}.
 \end{eqnarray*}
 Since	
 	$$ \lim\limits_{l \rightarrow	 + \infty}  \frac{1}{\phi(r)} r^{d(-\frac{1}{q}-\frac{1}{p})} \:_{r}\Vert f_{m+l} -f \Vert_{q,p} = 0,$$ we have
 	 $$ \Vert f_{m} - f \Vert_{(L^{q},L^{p})^{\phi}(\mathbb{R}^{d})}  \leq \epsilon .$$	
 This ends the proof.
\end{proof}

 The family of spaces $ (L^{q},L^{p})^{\phi}(\mathbb{R}^{d}) $ is decreasing with respect to the $ q $ power
 and increasing with respect to the $ p $ power. More precisely, we have the following proposition.
 \begin{proposition}\label{prop 1}
 Let  $ f \in L^{0}(\mathbb{R}^{d}). $    
 \begin{enumerate}
 \item If $1\leq q_{1}\leq q_{2} \leq p \leq\infty$ and $ \phi \in \mathcal{G}_{q_{2},p} $ then 
 $$\left\| f \right\|_{(L^{q_{1}},L^{p})^{\phi }(\mathbb{R}^{d})} \lesssim \left\|f\right\|_{(L^{q_{2}},L^{p})^{\phi }(\mathbb{R}^{d})},$$ and consequently, $(L^{q_{2}},L^{p})^{\phi}(\mathbb{R}^{d})\subset (L^{q_{1}},L^{p})^{\phi}(\mathbb{R}^{d}).$
 \item If $1\leq q\leq p_{1} \leq p_{2} \leq\infty$ and $ \phi \in \mathcal{G}_{q,p_{1}} $ then 
 $$\left\| f \right\|_{(L^{q},L^{p_{2}})^{\phi}(\mathbb{R}^{d})} \lesssim \left\|f\right\|_{(L^{q},L^{p_{1}})^{\phi}(\mathbb{R}^{d})}$$ and consequently, $(L^{q},L^{p_{1}})^{\phi}(\mathbb{R}^{d})\subset (L^{q},L^{p_{2}})^{\phi}(\mathbb{R}^{d}).$
 \end{enumerate}
 \end{proposition}
 
 \begin{proof}
Let $f \in L^{0}(\mathbb{R}^{d})$.\\
(1) Suppose that $1\leq q_{1}\leq q_{2} \leq p \leq\infty$ and $ \phi \in \mathcal{G}_{q_{2},p}$. It follows from Proposition \ref{prop 3} that
\begin{eqnarray*}
r^{d(-\frac{1}{q_{1}}-\frac{1}{p})}\: _{r}\left\|f\right\|_{q_{1},p} \lesssim r^{d(-\frac{1}{q_{2}}-\frac{1}{p})}\: _{r}\left\|f\right\|_{q_{2},p}, 
\end{eqnarray*}
for all $r>0$.\\
Hence, $\left\| f \right\|_{(L^{q_{1}},L^{p})^{\phi }(\mathbb{R}^{d})} \lesssim \left\|f\right\|_{(L^{q_{2}},L^{p})^{\phi }(\mathbb{R}^{d})}.$  \\
(2) Suppose that $1\leq q\leq p_{1} \leq p_{2} \leq\infty$ and $ \phi \in \mathcal{G}_{q,p_{1}}$. From Proposition \ref{prop 3}, we get
\begin{eqnarray*}
r^{-\frac{d}{p_{2}}}\: _{r}\left\|f\right\|_{q,p_{2}} \lesssim r^{-\frac{d}{p_{1}}}\: _{r}\left\|f\right\|_{q,p_{1}}, 
\end{eqnarray*}
for all $r>0$. Multiplying both sides of this inequality by $\frac{1}{\phi(r)} r^{-\frac{d}{q}}$, we obtain 
$$ r^{d(-\frac{1}{q}-\frac{1}{p_{2}})}\: _{r}\left\|f\right\|_{q,p_{2}} \lesssim  r^{d(-\frac{1}{q}-\frac{1}{p_{1}})}\: _{r}\left\|f\right\|_{q,p_{1}}.$$
Hence, $\left\| f \right\|_{(L^{q},L^{p_{2}})^{\phi}(\mathbb{R}^{d})} \lesssim \left\|f\right\|_{(L^{q},L^{p_{1}})^{\phi}(\mathbb{R}^{d})}.$
\end{proof}

The following proposition asserts that the characteristic functions of balls belong to generalized Fofana spaces. 

\begin{proposition}
 Let $1\leq q \leq  p \leq \infty$ and let $ \phi \in \mathcal{G}_{q,p}. $ Then 
 $$\frac{1}{\phi(r_{0})} \approx \left\| \chi_{B(0,r_{0})} \right\|_{(L^{q},L^{p})^{\phi }(\mathbb{R}^{d})}, $$ for every ball $ B(0,r_{0}). $
 \end{proposition}
\begin{proof}
Let $ r_{0} >0. $ It follows from Proposition \ref{prop 1} (2) that 
\begin{eqnarray*}
 \frac{1}{\phi(r_{0})} r_{0}^{-\frac{d}{q}} \left[\int_{\mathbb R^{d}} \chi_{B(0,r_{0})}(x)\chi_{B(0,r_{0})}(x)dx\right]^{\frac{1}{q}} &\leq &         \left\|\chi_{B(0,r_{0})} \right\|_{(L^{q},L^{+\infty})^{\phi}(\mathbb{R}^{d})}\\
 & \lesssim & \left\|\chi_{B(0,r_{0})}\right\|_{(L^{q},L^{p})^{\phi}(\mathbb{R}^{d})}.
\end{eqnarray*}
Thus, $ \frac{1}{\phi(r_{0})} \lesssim \left\| \chi_{B(0,r_{0})} \right\|_{(L^{q},L^{p})^{\phi }(\mathbb{R}^{d})}. $\\
We will prove the converse inequality in two steps. By definition of $\left\| \cdot \right\|_{(L^{q},L^{p})^{\phi }(\mathbb{R}^{d})},$ we have 
\begin{eqnarray*}
&&\left\| \chi_{B(0,r_{0})} \right\|_{(L^{q},L^{p})^{\phi }(\mathbb{R}^{d})} \\&=& \sup_{r>0} \frac{1}{\phi(r)} r^{d(-\frac{1}{q}-\frac{1}{p})} \left\|\left[\int_{\mathbb R^{d}} \chi_{B(0,r_{0})}(x)\chi_{B(\cdot,r)}(x)dx\right]^{\frac{1}{q}}\right\|_p. 
\end{eqnarray*}
\textit{First step:} For any real number $r$ such that $ r_{0} \leq r  $, we have
\begin{equation} \label{eqq 1}
r_{0}^{\frac{d}{q}}\phi(r_{0}) \leq C_{1} r^{\frac{d}{q}}\phi(r).
\end{equation}
It follows that
\begin{eqnarray*}
&&\sup_{r \ge r_{0}} \frac{1}{\phi(r)} r^{d(-\frac{1}{q}-\frac{1}{p})} \left\|\left[\int_{\mathbb R^{d}}  \chi_{B(0,r_{0})}(x)\chi_{B(\cdot,r)}(x)dx\right]^{\frac{1}{q}}\right\|_p \\&\leq& \sup_{r \ge r_{0}} \frac{1}{\phi(r)} r^{d(-\frac{1}{q}-\frac{1}{p})} \left\|\left[\int_{\mathbb R^{d}}  \chi_{B(0,r_{0})}(x)dx\right]^{\frac{1}{q}}\chi_{B(0,r+r_{0})}\right\|_p \\&\leq& 
 2^{d(\frac{1}{q}+\frac{1}{p})}
r^{\frac{d}{q}}_{0}\sup_{r \ge r_{0}} \frac{1}{\phi(r)} r^{d(-\frac{1}{q}-\frac{1}{p})} (r+r_{0})^{\frac{d}{p}}\\
&\leq& 2^{d(\frac{1}{q}+\frac{1}{p})}r^{\frac{d}{q}}_{0}\sup_{r \ge r_{0}} \frac{1}{\phi(r)} r^{d(-\frac{1}{q}-\frac{1}{p})} r^{\frac{d}{p}}(1+\frac{r_{0}}{r})^{\frac{d}{p}} \\
&\leq& 2^{d(\frac{1}{q}+\frac{2}{p})}r^{\frac{d}{q}}_{0}\sup_{r \ge r_{0}} \frac{1}{\phi(r)} r^{d(-\frac{1}{q}-\frac{1}{p})} r^{\frac{d}{p}}.
\end{eqnarray*} 
By inequality (\ref{eqq 1}), we get 
\begin{equation}\label{eq 2}
\sup_{r \ge r_{0}} \frac{1}{\phi(r)} r^{d(-\frac{1}{q}-\frac{1}{p})} \left\|\left[\int_{\mathbb R^{d}}  \chi_{B(0,r_{0})}(x)\chi_{B(\cdot,r)}(x)dx\right]^{\frac{1}{q}}\right\|_p \leq \frac{C^{'}}{\phi(r_{0})},
\end{equation}
where $C^{'}= 2^{d(\frac{1}{q}+\frac{2}{p})}C_{1}.  $\\
\textit{Second step:} For any real number $r$ such that $ r_{0} \geq r$, we have
\begin{equation}\label{eq 4}
r_{0}^{\frac{d}{p}}\phi(r_{0}) \leq C_{2} r^{\frac{d}{p}}\phi(r).
\end{equation}

\begin{eqnarray*}
&&\sup_{r \leq r_{0}} \frac{1}{\phi(r)} r^{d(-\frac{1}{q}-\frac{1}{p})} \left\|\left[\int_{\mathbb R^{d}}  \chi_{B(0,r_{0})}(x)\chi_{B(\cdot,r)}(x)dx\right]^{\frac{1}{q}}\right\|_p \\&\leq& \sup_{r \leq r_{0}} \frac{1}{\phi(r)} r^{d(-\frac{1}{q}-\frac{1}{p})} \left\|\left[\int_{\mathbb R^{d}}  \chi_{B(y,r)}(x)dx\right]^{\frac{1}{q}}\chi_{B(0,r+r_{0})}\right\|_p. \\
 &\leq& 2^{d(\frac{1}{q}+\frac{1}{p})}
\sup_{r \leq r_{0}} \frac{1}{\phi(r)} r^{d(-\frac{1}{q}-\frac{1}{p})}r^{\frac{d}{q}} (r+r_{0})^{\frac{d}{p}}\\
&\leq& 2^{d(\frac{1}{q}+\frac{1}{p})}r_{0}^{\frac{d}{p}}\sup_{r \leq r_{0}} \frac{1}{\phi(r)} r^{d(-\frac{1}{p})} (1+\frac{r}{r_{0}})^{\frac{d}{p}} \\
&\leq& 2^{d(\frac{1}{q}+\frac{2}{p})}r^{\frac{d}{p}}_{0}\sup_{r \leq r_{0}} \frac{1}{\phi(r)} r^{-\frac{d}{p}}.
\end{eqnarray*} 
By inequality (\ref{eq 4}), we get 
\begin{equation}\label{eq 3}
\sup_{r \ge r_{0}} \frac{1}{\phi(r)} r^{d(-\frac{1}{q}-\frac{1}{p})} \left\|\left[\int_{\mathbb R^{d}}  \chi_{B(0,r_{0})}(x)\chi_{B(\cdot,r)}(x)dx\right]^{\frac{1}{q}}\right\|_p \leq \frac{C^{''}}{\phi(r_{0})}
\end{equation}
where $C^{''}= 2^{d(\frac{1}{q}+\frac{2}{p})}C_{2}.  $\\
From (\ref{eq 2}) and (\ref{eq 3}), we deduce that $$\left\| \chi_{B(0,r_{0})} \right\|_{(L^{q},L^{p})^{\phi }(\mathbb{R}^{d})} \leq \frac{C}{\phi(r_{0})},$$ where $C=\max\{C^{'},  C^{''}  \}$.
\end{proof}
\section{Boundedness of the Hardy-Littlewood maximal operator}
It is proved in \cite{JF} that the Hardy-Littlewood maximal operator is bounded on classical Fofana spaces. It is also bounded on generalized Morrey spaces (see \cite{EN}). We shall prove the following result within the framework of generalized Fofana spaces. 
\begin{theorem}\label{theo 1}
 Let $1 < q\leq p \leq \infty$ and $\phi\in \mathcal{G}_{q,p}$. Assume that there is a constant $ C > 0 $ such that for any $ r >0 $, 
  \begin{equation}\label{eqq 2}
 \displaystyle\int_{r}^{\infty} \phi^{q}(t)t^{\frac{dq}{p}-1} dt \leq C   \phi^{q}(r)r^{\frac{dq}{p}}.
  \end{equation}
  Then $$\left\Vert Mf \right\Vert_{(L^{q},L^{p})^{\phi}(\mathbb{R}^{d})} \lsim  \left\Vert f  \right\Vert_{(L^{q},L^{p})^{\phi}(\mathbb{R}^{d})}, \quad f \in (L^{q},L^{p})^{\phi}(\mathbb{R}^{d}).$$
\end{theorem}
\begin{remark}
Theorem \ref{theo 1} in the case $1<q \leq \alpha \leq  p\leq \infty$ and $\phi (r) = r^{-\frac{d}{\alpha}}$ was proved in   \cite[Proposition 4.2 (1)]{JF}.  In the case $p=\infty$, it is nothing but Theorem 1 of \cite{EN}. \label{R10}
\end{remark}

The following result will be useful in the proof of Theorem \ref{theo 1}.

\begin{lemma}\label{lem 1}
Suppose that $1 < q\leq p \leq \infty $, $r >0$ and $\phi \in \mathcal{G}_{q,p}$.  Then, there exists a positive constant $ C $ such that for any positive integer $i$, we have 
$$  (2^{i}r)^{\frac{dq}{p}} \phi^{q}(2^{i+1}r) \leq C \displaystyle\int_{2^{i}r}^{2^{i+1}r} \phi^{q}(t)t^{\frac{dq}{p}-1} dt .$$
\end{lemma}
\begin{proof}
Suppose  that $1 < q\leq p \leq \infty$ and $\phi \in \mathcal{G}_{q,p}$. Let  $r >0$. We fix $i \geq 1$.
\\ For $ 2^{i}r \leq t <  2^{i+1}r $, we have  $  C^{'} \phi^{q}(2^{i+1}r)  \leq  \phi^{q}(t)$ according to (\ref{eq 1}). It follows that
\begin{eqnarray*}
\displaystyle\int_{2^{i}r}^{2^{i+1}r} \phi^{q}(t)t^{\frac{dq}{p}-1} dt &\ge & C^{'} \phi^{q}(2^{i+1}r)\displaystyle\int_{2^{i}r}^{2^{i+1}r} t^{\frac{dq}{p}}\frac{1}{t} dt \\ &\geq & C^{'}ln(2) (2^{i}r)^{\frac{dq}{p}}\phi^{q}(2^{i+1}r).
\end{eqnarray*}
Thus, $\quad  (2^{i}r)^{\frac{dq}{p}} \phi^{q}(2^{i+1}r) \leq C \displaystyle\int_{2^{i}r}^{2^{i+1}r} \phi^{q}(t)t^{\frac{dq}{p}-1} dt$\\ with $C=\frac{1}{C^{'}ln(2)}$.  
\end{proof}
We may now establish the proof of Theorem \ref{theo 1}.
\begin{proof}[Proof of Theorem \ref{theo 1}]
Let $1 < q\leq p \leq \infty$, $\phi\in \mathcal{G}_{q,p}$ and $f \in (L^{q},L^{p})^{\phi}(\mathbb{R}^{d}).$\\
Since the desired result is already established in the case $p=\infty$ (see Remark \ref{R10} ), we suppose that $ p < \infty $.\\
Let $r>0$ and $y \in \mathbb{R}^{d}$. Following \cite[Theorem 2.12]{GF}, we may write
\begin{eqnarray}\label{rel 8}
\int_{\mathbb R^{d}} (Mf)^{q}(x) \chi_{B(y,r)}(x) dx \lsim \int_{\mathbb R^{d}} \vert f(x) \vert^{q} (M \chi_{B(y,r)})(x)dx.
\end{eqnarray}
Then, proceeding as in the proof of Theorem 1 in \cite{CF}, we obtain
$$_{r}\left\|Mf\right\|^{q}_{q,p} \lsim \: _{2r}\left\|f\right\|^{q}_{q,p} +  \sum^{\infty}_{i=1} \frac{1}{(2^{i-1})^{d}}  \: _{2^{i+1}r}\left\|f\right\|^{q}_{q,p}.$$
On the one hand, we have
\begin{eqnarray*}
_{2r}\left\|f\right\|^{q}_{q,p} & = & \frac{\phi^{q}(2r)(2r)^{dq(-\frac{1}{q}-\frac{1}{p})}}{\phi^{q}(2r)(2r)^{dq(-\frac{1}{q}-\frac{1}{p})}}   \: _{2r}\left\|f\right\|^{q}_{q,p}\\
&\leq& \frac{\phi^{q}(2r)}{(2r)^{dq(-\frac{1}{q}-\frac{1}{p})}}\left\Vert f \right\Vert^{q}_{(L^{q},L^{p})^{\phi }(\mathbb{R}^{d})}\\
&\leq& C \phi^{q}(r) r^{dq(\frac{1}{q}+\frac{1}{p})} \left\Vert f \right\Vert^{q}_{(L^{q},L^{p})^{\phi }(\mathbb{R}^{d})}.
\end{eqnarray*}
Hence,
\begin{equation}\label{rel 1}
_{2r}\left\|f\right\|^{q}_{q,p} \lsim \phi^{q}(r) r^{dq(\frac{1}{q}+\frac{1}{p})} \left\Vert f \right\Vert^{q}_{(L^{q},L^{p})^{\phi }(\mathbb{R}^{d})}.
\end{equation}
On the other hand, thanks to (\ref{rel 1}), we get
\begin{eqnarray*}
& &\sum^{\infty}_{i=1} \frac{1}{(2^{i-1})^{d}}  \: _{2^{i+1}r}\left\|f\right\|^{q}_{q,p} \\
&\lsim & \sum^{\infty}_{i=1} \frac{1}{(2^{i-1})^{d}} \phi^{q}(2^{i+1}r) (2^{i+1}r)^{dq(\frac{1}{q}+\frac{1}{p})} \left\Vert f \right\Vert^{q}_{(L^{q},L^{p})^{\phi }(\mathbb{R}^{d})}.  
\end{eqnarray*}
It follows that
\begin{eqnarray*}
& &\sum^{\infty}_{i=1} \frac{1}{(2^{i-1})^{d}}  \: _{2^{i+1}r}\left\|f\right\|^{q}_{q,p} \\
&\lsim &\left\Vert f \right\Vert^{q}_{(L^{q},L^{p})^{\phi }(\mathbb{R}^{d})}r^{dq(\frac{1}{q}+\frac{1}{p})}  \sum^{\infty}_{i=1} 2^{\frac{idq}{p}} \phi^{q}(2^{i+1}r)  \\
&\lsim &\left\Vert f \right\Vert^{q}_{(L^{q},L^{p})^{\phi }(\mathbb{R}^{d})}r^{dq(\frac{1}{q}+\frac{1}{p})} r^{-\frac{dq}{p}} \sum^{\infty}_{i=1} (2^{i}r)^{\frac{dq}{p}} \phi^{q}(2^{i+1}r) .
\end{eqnarray*}
By applying Lemma \ref{lem 1}, the previous inequality becomes
\begin{eqnarray*}
&&\sum^{\infty}_{i=1} \frac{1}{(2^{i-1})^{d}}  \: _{2^{i+1}r}\left\|f\right\|^{q}_{q,p} \\
&\lsim & \left\Vert f \right\Vert^{q}_{(L^{q},L^{p})^{\phi }(\mathbb{R}^{d})}r^{dq(\frac{1}{q}+\frac{1}{p})} r^{-\frac{dq}{p}} \sum^{\infty}_{i=1} \displaystyle\int_{2^{i}r}^{2^{i+1}r} \phi^{q}(t)t^{\frac{dq}{p}-1} dt. 
\end{eqnarray*}
Since $$\sum\limits^{\infty}_{i=1} \displaystyle\int_{2^{i}r}^{2^{i+1}r} \phi^{q}(t)t^{\frac{dq}{p}-1} dt \leq  \displaystyle\int_{2r}^{\infty} \phi^{q}(t)t^{\frac{dq}{p}-1} dt \leq \displaystyle\int_{r}^{\infty} \phi^{q}(t)t^{\frac{dq}{p}-1} dt, $$
we get $$\sum^{\infty}_{i=1} \frac{1}{(2^{i-1})^{d}}  \: _{2^{i+1}r}\left\|f\right\|^{q}_{q,p} \lsim  \left\Vert f \right\Vert^{q}_{(L^{q},L^{p})^{\phi }(\mathbb{R}^{d})}r^{dq(\frac{1}{q}+\frac{1}{p})} r^{-\frac{dq}{p}} \displaystyle\int_{r}^{\infty} \phi^{q}(t)t^{\frac{dq}{p}-1} dt.  $$
It follows from (\ref{eqq 2}) that, 
\begin{equation}\label{rel 2}
\sum^{\infty}_{i=1} \frac{1}{(2^{i-1})^{d}}  \: _{2^{i+1}r}\left\|f\right\|^{q}_{q,p} \lsim \left\Vert f \right\Vert^{q}_{(L^{q},L^{p})^{\phi }(\mathbb{R}^{d})}r^{dq(\frac{1}{q}+\frac{1}{p})} \phi^{q}(r).
\end{equation}
From (\ref{rel 1}) and (\ref{rel 2}),  we deduce that $$_{r}\left\|Mf\right\|^{q}_{q,p} \lsim \left\Vert f \right\Vert^{q}_{(L^{q},L^{p})^{\phi }(\mathbb{R}^{d})}r^{dq(\frac{1}{q}+\frac{1}{p})} \phi^{q}(r).$$
It follows that, 
\begin{equation} \label{eqq 3} 
\frac{1}{\phi^{q}(r)}r^{dq(-\frac{1}{q}-\frac{1}{p})} \: _{r}\left\|Mf\right\|^{q}_{q,p} \lsim \left\Vert f \right\Vert^{q}_{(L^{q},L^{p})^{\phi }(\mathbb{R}^{d})}.
\end{equation}
We obtain the desired result by taking the supremum over all $ r > 0 $ in
the left hand side of (\ref{eqq 3}).
\end{proof}

\end{document}